\documentclass[12pt, reqno]{amsart}
\usepackage{amssymb,latexsym,amsmath,amscd,amsthm,graphicx, color}
\usepackage[all]{xy}
\usepackage{pgf,tikz}
\usepackage{mathrsfs}
\usetikzlibrary{arrows}
\usepackage[left=0.6 in, top=0.6 in, right=0.6 in, bottom=0.5 in]{geometry}
\usepackage{enumerate}
\usepackage[linkcolor=blue, urlcolor=blue, citecolor=blue,
colorlinks, bookmarks]{hyperref}
\definecolor{uuuuuu}{rgb}{0.26666666666666666,0.26666666666666666,0.26666666666666666}
\definecolor{xdxdff}{rgb}{0.49019607843137253,0.49019607843137253,1.}
\definecolor{ffqqqq}{rgb}{1.,0.,0.}
\definecolor{ffqqqq}{rgb}{1.,0.,0.}
\definecolor{ffxfqq}{rgb}{1.,0.4980392156862745,0.}

\definecolor{uuuuuu}{rgb}{0.26666666666666666,0.26666666666666666,0.26666666666666666}
\definecolor{qqwuqq}{rgb}{0.,0.39215686274509803,0.}
\definecolor{zzttqq}{rgb}{0.6,0.2,0.}
\definecolor{xdxdff}{rgb}{0.49019607843137253,0.49019607843137253,1.}
\definecolor{qqqqff}{rgb}{0.,0.,1.}
\definecolor{cqcqcq}{rgb}{0.7529411764705882,0.7529411764705882,0.7529411764705882}
\definecolor{sqsqsq}{rgb}{0.12549019607843137,0.12549019607843137,0.12549019607843137}

\theoremstyle{plain}

\newtheorem{theorem}[subsection]{Theorem}

\newtheorem{lemma}[subsection]{Lemma}
\newtheorem{defi}[subsection]{Definition}
\newtheorem{prop}[subsection]{Proposition}

\theoremstyle{definition}
\newtheorem{defi1}[subsection]{Definition}

\newtheorem{example}[subsection]{Example}

\newtheorem{remark}[subsection]{Remark}

\newtheorem{method}[subsection]{Methodology}

\newcommand{\uu}{\cup}% union
\newcommand{\UU}{\bigcup}% big union
\newcommand{\ci}{\subseteq}% contained in with equality
\newcommand{\sci}{\subset}% strictly contained in
\newcommand{\set}[1]{\{#1\}}% set
\newcommand{\ga}{\gamma}
\newcommand{\gb}{\beta}

\newcommand{\gk}{\kappa}

\newcommand{\gq}{\theta}

\newcommand{\tit}{\textit}% text italic

\newcommand{\C}[1]{\mathcal{#1}}% Euler Script - only caps, use as \C{A}
\newcommand{\D}[1]{\mathbb{#1}}% Doubled - blackboard bold - only caps, uas as \D{A}
\newcommand{\te}{\text}% same as \mathrm command.

\newcommand{\pa}{\partial}

\begin{document}
	\title{Constrained Quantization for Uniform Distributions with Two Constraint Families}
	
	\author{$^1$Pavjeet Singh}
	\author{$^2$S. K. Katiyar}
	\author{$^3$Megha Pandey}
	\author{$^4$Mrinal Kanti Roychowdhury}

	\address{$^{1, 2}$Department of Mathematics\\
		Dr B R Ambedkar National Institute of Technology Jalandhar\\
		Jalandhar Punjab, India 144011.}

	\address{$^{3}$ School of Mathematics \\
		Northwest University Xi'an\\
		Shaanxi Province, 710069, PR China.}
	\address{$^{4}$School of Mathematical and Statistical Sciences\\
		University of Texas Rio Grande Valley\\
		1201 West University Drive\\
		Edinburg, TX 78539-2999, USA.}

	\email{$^1$pavjeetsingh23@gmail.com, $^2$sbhkatiyar@gmail.com}
	\email{$^3$meghapandey1071996@gmail.com, $^4$mrinal.roychowdhury@utrgv.edu}

	
	\date{}
	\maketitle
	
	\pagestyle{myheadings}\markboth{Singh, Katiyar, Pandey, Roychowdhury}{Constrained quantization for a uniform distribution with respect to family of constraints}
	\begin{abstract}
		In this paper, we first consider a family of constraints given by straight lines. 
		For a uniform probability distribution, we determine the constrained optimal sets of 
		$n$-points and the corresponding $n$th constrained quantization errors for all positive 
		integers $n$. In addition, we calculate the constrained quantization dimension and the 
		constrained quantization coefficient with respect to this family of constraints. Next, 
		we turn to another family of constraints, consisting of concentric circles. For the same 
		probability distribution, we present a methodology to compute the constrained optimal sets 
		of $n$-points and the corresponding $n$th constrained quantization errors for all positive 
		integers $n$. Finally, we conclude the paper with a summary of the results and a discussion 
		of future research directions.
	\end{abstract}
	
	\section{Introduction}
	There are numerous practical uses for quantization theory, especially when a huge amount of data needs to be condensed into a limited number of distinct data points. Engineering technologies like data compression and signal processing are the source of quantization issues. As a result, quantization-related mathematical conclusions have a wide range of applications in several scientific fields, including data compression, signal processing, information theory, and communications (see \cite{GG, GL, GN, GL1, P, Z1, Z2}). Finding the best approximation of a probability distribution $\nu$ using a discrete probability distribution $Q$ with a given number $n$ of supporting points (also known as the optimal set of $n$-points) is known as quantization in the context of probability distributions. Recent work by Pandey and Roychowdhury \cite{PR2, PR1} developed the concept of constrained quantization. Constrained quantization is a specific instance of unconstrained quantization, which is commonly referred to as quantization in the literature. For quantization without constraints, see to \cite{DR,DFG,GL2, GL3, KNZ}. 
	\begin{defi1}\label{Vr}
		Let $\nu $ be a Borel probability measure on $\D R^2$. Let $\set{C_j\ci \D R^2: j\in \D N}$ be a family of closed sets with $C_1$ nonempty and $\ga\ci \UU_{j=1}^nC_j$ be a locally finite (i.e., intersection of $\ga$ with any bounded subset of $\D R^2$ is finite) subset of $\D R^2$. This implies that $\ga$ is countable and closed. Then, for $n\in \mathbb{N}$, the \tit {$n$th constrained quantization error} for $\nu$ with respect to the family of constraints $\set{C_j\ci \D R^2: j\in \D N}$, is defined by
		\begin{equation} \label{EqVr}
			R_{n}:=R_{n}(\nu)=\inf \Big\{\int \mathop{\min}\limits_{a\in\ga} d(x, a)^2 d\nu(x) : \ga \ci \UU_{j=1}^nC_j, ~ 1\leq  \text{card}(\ga) \leq n \Big\},
		\end{equation}
		where $\te{card}(A)$ represents the cardinality of the set $A$.
	\end{defi1}
	The number 
	\begin{equation*}
		R(\nu; \ga):= \int \mathop{\min}\limits_{a\in\ga} d(x, a)^2 d\nu(x)
	\end{equation*}
	is called the distortion error for $\nu$ with respect to a set $\ga \ci \D R^2$. 
	\begin{lemma}\emph{\cite{PR1}}
		If $\int d(x, 0)^2 d\nu(x)<\infty$ is satisfied, then the infimum in \eqref{EqVr} is achieved.
	\end{lemma}
	A set $ \ga \ci \mathop{\UU}\limits_{j=1}^n C_j$ for which the infimum in  \eqref{EqVr} is attained is called a \tit{constrained optimal set of $n$-points} for $\nu$.  
	
	\begin{remark}\label{cardinalityalphan}
		Let $\nu$ be a Borel probability measure defined on $\mathbb{R}^2$, such that the support of $\nu$ contains at least $n$ distinct points. In the case of unconstrained quantization, it is well-known that an optimal set of $n$-points (also referred to as an optimal set of $n$-means) consists of exactly $n$ elements (see \cite{GL, PR1}). However, this property does not generally hold in the setting of constrained quantization, i.e., if the support of $\nu$ contains at least $n$ elements that does not imply that there exists a constrained optimal set containing exactly $n$ elements (see \cite{PR1}).
	\end{remark}

	\begin{defi}
		Given a finite subset $\ga\sci \D R^2$, the Voronoi region generated by $a\in \ga$ is defined by
		\[M(a|\ga)=\set{x \in \D R^2 : \lVert x-a\rVert =\min_{b \in \ga}\lVert x-b\rVert}\]
		i.e., the Voronoi region generated by $a\in \ga$ is the set of all elements $x$ in $\D R^2$ such that $a$ is the nearest element to $x$ in $\ga$, and the set $\set{M(a|\ga) : a \in \ga}$ is called the Voronoi diagram or Voronoi tessellation of $\D R^2$ with respect to $\ga$. In unconstrained quantization, the elements in an optimal set of $n$-points are the conditional expectations in their own Voronoi regions. Hence, in unconstrained quantization an optimal set of $n$-points is referred to as an optimal set of $n$-means.
	\end{defi}
	
	\begin{defi}\cite{PR2,PR1}
		Let $R_{n}(\nu)$ be a strictly decreasing sequence, and write $R_{\infty}(\nu):=\mathop{\lim}\limits_{n\to \infty} R_{n}(\nu)$. 
		Then, the number $D(\nu)$ defined by 
		\[D(\nu):=\mathop{\lim}\limits_{n\to \infty} \frac{2\log n}{-\log (R_{n}(\nu)-R_{\infty}(\nu))},  \]
		if it exists, is called the \tit{constrained quantization dimension} of $\nu$. 
		For any $\gk>0$, the  number  
		\begin{equation} \label{eq00100} \lim_{n\to \infty} n^{\frac 2 \gk}  (R_{n}(\nu)-R_{\infty}(\nu)),\end{equation} if it exists, is called the \tit{$\gk$-dimensional constrained quantization coefficient} for $\nu$.
	\end{defi}
	
	\begin{prop}\emph{\cite{PR1}}
		If $\kappa$-dimensional constrained quantization coefficient is finite and positive, then the constrained quantization dimension of $\nu$ exists and $\gk$ equals $D(\nu)$.
	\end{prop}
	
	\begin{remark}
		It is important to note that, in the definition of constrained quantization dimension, the sequence of quantization errors $\{R_n(\nu)\}_{n \in \mathbb{N}}$ is strictly decreasing. This assumption holds naturally in the setting of unconstrained quantization. However, in constrained quantization problems (see \cite{PR1}), this strict monotonicity may fail. In particular, as discussed in Remark \ref{cardinalityalphan}, there may exists an index $N\in \mathbb{N}$ such that optimal set $\ga_n$ remains unchanged for all $n\geq N$, i.e., $\ga_n = \ga_N$ for all $n \geq N$. In this case, increasing the number of points beyond $N$ does not improve the approximation, and the quantization error becomes constant from that point onward, i.e.,
		\[R_n(\nu) = R_N(\nu) = R_\infty(\nu) \quad \text{for all } n \geq N.\]
		As a result, the difference $R_n(\nu) - R_\infty(\nu)$ vanishes for large $n$, and thus the quantization dimension, defined by
		\[D(\nu) := \lim_{n \to \infty} \frac{2\log n}{-\log (R_n(\nu) - R_\infty(\nu))},\]
		does not exist. 
	\end{remark}

	%
	%With respect to a finite set $\ga \sci \D R^2$, by the \tit{Voronoi region} of an element $a\in \ga$, it is meant the set of all elements in $\D R^2$ which are nearest to $a$ among all the elements in $\ga$, and is denoted by $M(a|\ga)$.  
	%In unconstrained quantization the elements in an optimal set of $n$-points are the conditional expectations in their own Voronoi regions. Hence, in unconstrained quantization an optimal set of $n$-points \textcolor{blue}{is} referred to as an \tit{optimal set of $n$-means}. 
	% 

	\subsection{Motivation and work done} 
	Research on \emph{unconstrained quantization} has been widely developed and successfully 
	applied to various probability distributions, such as the uniform and self-similar distributions (see \cite{P1, RR, R1, R2, R3, RS}). 
	These studies have demonstrated the efficiency of quantization techniques in many settings. 
	However, in practical applications, additional \emph{constraints} often arise---for example, 
	spatial limitations in wireless communication systems or requirements of symmetry in physical 
	processes. Such constraints complicate the quantization process, since traditional methods may 
	fail to handle them effectively without a substantial increase in distortion. 
	
	This makes the study of \emph{constrained quantization} both important and challenging. 
	For instance, in cancer radiation therapy, beams of radiation must be carefully directed toward 
	tumor cells while sparing healthy tissue. In this context, the patient’s body can be viewed as the 
	support of a probability measure, with the tumor cells are the possible locations of 
	optimal points under constraints (\cite{ bangert2013analytical, shusharina2018clinical}).  
	
	The idea of constrained quantization was first introduced by Pandey and Roychowdhury for 
	uniform distributions under a single constraint (see \cite{PR1}), and later extended to families of constraints in 
	the case of the standard Cantor distribution (see \cite{PR2}).  
	
	In the present paper, we advance this line of study by investigating constrained quantization 
	for a \emph{uniform distribution} with respect to two distinct families of constraints:
	\begin{enumerate}
		\item [] \textbf{Lines:} 
		\begin{equation} \label{Megha000} 
			C_j: = \{(x,y) : -\tfrac{1}{j} \leq x \leq 1, \; y = x + \tfrac{1}{j} \}, 
			\quad j \in \mathbb{N}.
		\end{equation} \label{Megha00} 
		\item [] \textbf{Circles:} 
		\begin{equation} \label{Megha001}
			C_j:= \{(x,y) : x^2 + y^2 = \tfrac{1}{j^2} \}, 
			\quad j \in \mathbb{N}.
		\end{equation}
	\end{enumerate}

	\subsection{Delineation} 
	The structure of the paper is outlined as follows. Let $\nu$ denote the uniform distribution with support the closed interval $[0,1]$. In the next section, we present the required preliminaries and notations that will serve as the foundation for the subsequent analysis. The main results are contained in Section~\ref{section3} and Section~\ref{section4}. 
	Section~\ref{section3} is divided into two subsections. In the first subsection, we determine the constrained optimal sets of $n$-points together with the corresponding $n$th constrained quantization errors for the uniform distribution $\nu$ with respect to the family of constraints specified in~\eqref{Megha000}. The second subsection focuses on the quantization dimension and the quantization coefficient for the same distribution with respect to the same family of constraints.  
	Section~\ref{section4} deals with the determination of the constrained optimal sets of $n$-points and the associated constrained quantization errors for the distribution $\nu$ under the family of constraints defined in~\eqref{Megha001}. 
	
	Finally, Section~\ref{conclusionfuturework} concludes the paper with a summary of the findings and an outline of possible avenues for future research.
	
	\section{Preliminaries}
	In this section, we provide certain fundamental concepts and notations that we have utilized in this paper. We write \[\rho((x_1, y_1), (x_2, y_2)):=(x_1 - x_2)^2 +(y_1-y_2)^2,\] for any two elements $(x_1, y_1)$ and $(x_2, y_2)$ in $\D R^2$. which provides the squared Euclidean distance between the two elements $(x_1, y_1)$ and $(x_2, y_2)$.
	\par
	Let $a$ and $b$ be two elements that belong to an optimal set of $n$-points for some positive integer $n$. Then, $a$ and $b$  are called \tit{adjacent elements} if they have a common boundary in their own Voronoi regions. Let $c$ be an element on the common boundary of the Voronoi regions of the adjacent elements $a$ and $b$.
	Since the common boundary of the Voronoi regions of any two adjacent elements is the perpendicular bisector of the line segment joining the elements, we have
	\[\rho(a, c)-\rho(b, c)=0. \]
	We call such an equation a \tit{canonical equation}. 
	Notice that any element $x\in \D R$ can be identified as an element $(x, 0)\in \D R^2$. Thus, 
	\[\rho: \D R \times \D R^2 \to [0, \infty) \te{ such that } \rho(x, (a, b))=(x-a)^2 +b^2,\]
	where $x\in \D R$ and $(a, b) \in \D R^2$, defines a nonnegative real-valued function on $\D R \times \D R^2$. 
	Let $\pi: \D R^2 \to \D R$ such that $\pi(a, b)=a$ for any $(a, b) \in \D R^2$ denote the projection mapping. 
	Recall that $\nu$ is a Borel probability measure on $\D R$ which is uniform on its support the closed interval $[0, 1]$. Hence, the probability density function $f$ for $\nu$ is given by 
	\[f(x)=\left\{\begin{array}{cc}
		1 & \te{ if } 0\leq x\leq 1,\\
		0 & \te{ otherwise}.
	\end{array}\right.
	\]
	Hence, we have $d\nu(x)=\nu(dx)=f(x) dx$ for any $x\in \D R$, where $d$ stands for differential. For a random variable $X$ with distribution $\nu$, let $E(X)$ represents the expected value, and $V:=V(X)$ represent the variance of $X$. 
	In this paper, in Section~\ref{section3}, we investigate the constrained quantization  for the probability measure $\nu$ with respect to the family of constraints given by 
	\begin{equation} \label{eq000} C_j=\set{(x, y) : -\frac 1 j\leq x\leq 1 \te{ and } y=x+\frac 1 j } \te{ for all } j\in\D N, 
	\end{equation}
	i.e., the constraints $C_j$ are the line segments joining the points $(-\frac 1 j, 0)$ and $(1, 1+\frac 1 j)$ which are parallel to the line $y=x$.
	The perpendicular on a constraint $C_j$ passing through a point $(x, x+\frac 1 j)\in C_j$ intersects the real line at the point $(2x+\frac 1 j, 0)$ where $-\frac 1 j\leq x\leq 1$; and it intersects $J=[0,1]$ if 
	$0\leq 2x+\frac 1 j\leq 1$, i.e., if 
	\begin{equation} \label{eq0000} -\frac 1 {2j}\leq x\leq\frac 1 2-\frac 1 {2j}.
	\end{equation} 
	Thus, for all $j\in \D N$, there exists a one-one correspondence between the elements $(x, x+\frac 1 j)$ on $C_j$ and the elements $2x+\frac 1 j\in \D R$ if $-\frac 1 j\leq x\leq 1$. Thus, for all $j\in \D N$, there exist bijective functions $U_j$ such that
	\begin{equation} \label{eq0001} 
		U_j(x, x+\frac 1j)=2x+\frac 1 j \te{ and } U_j^{-1}(x)=\Big(\frac 1 2 (x-\frac 1 j), \frac 1 2 (x-\frac 1 j)+\frac 1 j\Big), \end{equation}
	where $-\frac 1 j\leq x\leq 1$. Then, in Section~\ref{section4}, we investigate the constrained quantization  for the probability measure $\nu$ with respect to the family of constraints given by 
	\begin{equation} \label{eq000} C_j=\set{(x, y) : x^2+y^2=\frac 1{j^2} } \te{ for all } j\in\D N, 
	\end{equation}
	i.e., the constraints $C_j$ are the concentric circles with center $(0, 0)$ and radius $\frac 1 {j}$.  
	
	\section{Constrained quantization with respect to the family of constraints $C_j: = \set{(x,y) : -\tfrac{1}{j} \leq x \leq 1, \; y = x + \frac{1}{j}}~ \text{\normalfont for }  j\in \D N.$} \label{section3}
	The results in this section is given in the following two subsections. 
	
	\subsection{Constrained optimal sets of $n$-points and the $n$th constrained quantization errors}
	In this subsection, we calculate the constrained optimal sets of $n$-points and the $n$th constrained quantization errors for all $n\in\D N$ with respect to the family of constraints 
	$C_j = \set {(x,y) : -\frac{1}{j} \leq x \leq 1, \; y = x + \frac{1}{j} }$ for  $j\in \D N. $
	Let us first give the following lemma.
	\begin{lemma} \label{lemma0}
		Let $\ga_n\ci \mathop{\uu}\limits_{j=1}^n C_j$ be a constrained optimal set of $n$-points for $\nu$ such that
		\[\ga_n:=\set{(a_j, b_j) : 1\leq j\leq n},\]
		where $a_1<a_2<a_3<\cdots<a_n$. Then, $\ga_n\ci C_n$ and  $(a_j, b_j)=U_n^{-1}(E(X :  X\in \pi(M((a_j, b_j)|\ga_n))))$,
		where $M((a_j, b_j)|\ga_n)$ are the Voronoi regions of the elements $(a_j, b_j)$ with respect to the set $\ga_n$ for $1\leq j\leq n$.
	\end{lemma}
	
	\begin{proof}
		Let $\ga_n:=\set{(a_j, b_j): 1\leq j\leq n}$, as given in the statement of the lemma, be a constrained optimal set of $n$-points. Take any $(a_q, b_q)\in \ga_n$. Since $\ga_n\ci \mathop{\uu}\limits_{j=1}^n C_j$, we can assume that $(a_q, b_q) \in C_t$, i.e.,  $b_q=a_q+\frac 1 t$ for some $1\leq t\leq n$. Since the Voronoi region of $(a_q, b_q)$, i.e., $M((a_q, b_q)|\ga_n)$ has positive probability, we can assume that $M((a_q, b_q)|\ga_n)$ contains a subinterval $[a, b]$ from  the support $[0, 1]$ of $\nu$, where $0\leq a<b\leq 1$.  Hence, the distortion error contributed by $(a_q, b_q)$ in its Voronoi region $M((a_q, b_q)|\ga_n)$ is given by
		\begin{align*}
			&\int_{M((a_q, b_q)|\ga_n)}\rho(x, (a_q, b_q)) \,d\nu\\
			&=\frac{1}{3} (b-a) \Big(a^2-3 (a+b) a_q+a b+3 a_q^2+b^2+3 b_q^2\Big)\\
			&=\frac{1}{3} (b-a) \Big(a^2-3 (a+b) a_q+a b+3 \Big(a_q+\frac{1}{t}\Big)^2+3 a_q^2+b^2\Big).
		\end{align*}
		The above expression is minimum if $a_q=\frac{a t+b t-2}{4 t}$. Now, putting $a_q=\frac{a t+b t-2}{4 t}$, we have the above distortion error as 
		\[\frac{(b-a) \left(t^2 \left(5 a^2+2 a b+5 b^2\right)+12 t (a+b)+12\right)}{24 t^2}.\]
		Since $1\leq t\leq n$, the above distortion error is minimum if $t=n$. Thus, for $t=n$, we see that $(a_q, b_q) \in C_n$, and  
		\[a_q=\frac 12(\frac {a+b}2-\frac 1{n}),\te{ and } b_q=a_q+\frac 1 n=\frac 12(\frac {a+b}2-\frac 1{n})+\frac 1 n,\]
		which by \eqref{eq0001} yields that 
		\[(a_q, b_q)=U_n^{-1}\Big(\frac {a+b}2\Big)=U_n^{-1}(E(X :  X\in \pi(M((a_j, b_j)|\ga_n)))).\]
		Since $(a_q, b_q)\in \ga_n$ is chosen arbitrarily, the proof of the lemma is complete.  
	\end{proof}
	
	\begin{remark} \label{remM1} 
		By \eqref{eq0000} and \eqref{eq0001}, and Lemma~\ref{lemma0}, we can conclude that all the elements in a constrained optimal set of $n$-points must lie on $C_n$ between the two elements $U_n^{-1}(0,0)$ and $U_n^{-1}(1,0)$, i.e., between the two elements $(-\frac{1}{2 n},\frac{1}{2 n})$ and $(\frac{n-1}{2 n},\frac{n+1}{2 n})$. If this fact is not true, then the constrained quantization error can be strictly reduced by moving the elements in the constrained optimal set between the elements $(-\frac{1}{2 n},\frac{1}{2 n})$ and $(\frac{n-1}{2 n},\frac{n+1}{2 n})$ on $C_n$, in other words, the $x$-coordinates of all the elements in a constrained optimal set of $n$-points must lie between the two numbers $-\frac{1}{2 n}$ and $\frac{n-1}{2 n}$  (see Figure~\ref{Fig}).
	\end{remark} 
	
	\begin{lemma} \label{lemmaM1} 
		Let $\ga_n$ be a constrained optimal set of $n$-points for $\nu$. Then, $U_n(\ga_n)$ is an unconstrained optimal set of $n$-means for $\nu$. 
	\end{lemma}
	\begin{proof} 
		By Lemma~\ref{lemma0}, $\ga_n\ci C_n$ for all $n\in \D N$. Let $\ga_n:=\set{(a_j, b_j) : 1\leq j\leq n}$ be a constrained optimal set of $n$-points for $\nu$ such that $a_1<a_2<\cdots<a_n$. Then, by Remark~\ref{remM1}, we have $-\frac 1 {2n}\leq a_1<a_2<\cdots<a_n\leq \frac{n-1}{2 n}$.  
		Moreover, as $(a_j, b_j)\in C_n$, we have $b_j=a_j+\frac 1 n$ for all $1\leq j\leq n$. 
		\par 
		Notice that the boundary of the Voronoi regions of the adjacent elements $(a_j, b_j)$ and $(a_{j+1}, b_{j+1})$ intersects the support of $\nu$ at the elements $(a_j+a_{j+1}+\frac 1 n, 0)$ for $1\leq j\leq n-1$. Hence, the distortion error due to the set $\ga_n$ is given by
		\begin{align*}
			&R(\nu;\ga_n)=\int_{\D R} \min_{a\in \ga_n}\rho(x, a)d\nu(x)\\
			&=\int_0^{a_1+a_2+\frac 1 n}   \rho(x, (a_1, a_1+\frac 1n)) \, dx+\sum _{i=2}^{n-1} \int_{a_{i-1}+a_i+\frac 1 n}^{a_{i}+a_{i+1}+\frac 1 n}\rho(x, (a_i, a_i+\frac 1 n))\, dx\\
			&\qquad +\int_{a_{n-1}+a_n+\frac 1 n}^1\rho(x, (a_n, a_n+\frac 1 n)) \, dx.
		\end{align*}
		Since $R(\nu;\ga_n)$ gives the optimal error and is differentiable with respect to $a_i$ for all $1\leq i\leq n$, we have $\frac{\pa}{\pa a_i} R(\nu;\ga_n)=0$ implying
		\[\frac 1 n+2a_1=a_2-a_1=a_3-a_2=\cdots=a_n-a_{n-1}=1-\frac 1 n-2 a_n.\]
		Then, we can assume that there is a constant $d$ depending on $n$, such that 
		\begin{equation} \label{eqM1} \frac 1 n+2a_1=a_2-a_1=a_3-a_2=\cdots=a_n-a_{n-1}=1-\frac 1 n-2 a_n=d
		\end{equation} 
		yielding 
		\[a_2=d+a_1, \, a_3=2d+a_1, \, a_4=3d+a_1, \, \cdots, \, a_n=(n-1) d+a_1, \]
		i.e., 
		\begin{equation} \label{eqM2} a_j=(j-1)d +a_1 \te{ for } 2\leq j\leq n.
		\end{equation} 
		Again, by \eqref{eqM1}, we have 
		\begin{equation} \label{eqM3} a_1=\frac 12(d-\frac 1 n) \te{ and } a_n=\frac 1 2(1-\frac 1 n-d).
		\end{equation} 
		Putting the above values of $a_1$ and $a_n$ in the expression $a_n=(n-1) d+a_1$, and then upon simplification, we have $d=\frac 1 {2n}$. Putting the values of $d$ by \eqref{eqM2} and \eqref{eqM3}, we have 
		\begin{equation}
			a_j=\frac{2j-3}{4n} \te{ for } 1\leq j\leq n.
		\end{equation} 
		Then, notice that $a_j+a_{j+1}+\frac 1 n=\frac j n$ for $1\leq j\leq n-1$.
		Hence, by Lemma~\ref{lemma0}, we have 
		\[(a_j, b_j)= U_n^{-1}(E(X :  X\in [\frac{j-1}{n}, \frac j n]))=U_n^{-1}(\frac {2j-1}{2n}) \te{ for } 1\leq j\leq n.\]
		We know that for the uniform distribution $\nu$, the unconstrained optimal set of $n$-means (see \cite{RR}) is given by 
		\[\left\{\frac{2j-1}{2n} : 1\leq j\leq n\right\}.\]
		Since 
		\[U_n(\ga_n) =\set{U_n(a_j, b_j) : 1\leq j\leq n}=\left\{\frac{2j-1}{2n} : 1\leq j\leq n\right\},\]
		the proof of the lemma is complete. 
	\end{proof}

	\begin{figure}
		\vspace{-0.25 in}
		\centerline{\includegraphics[width=6 in, height=4in]{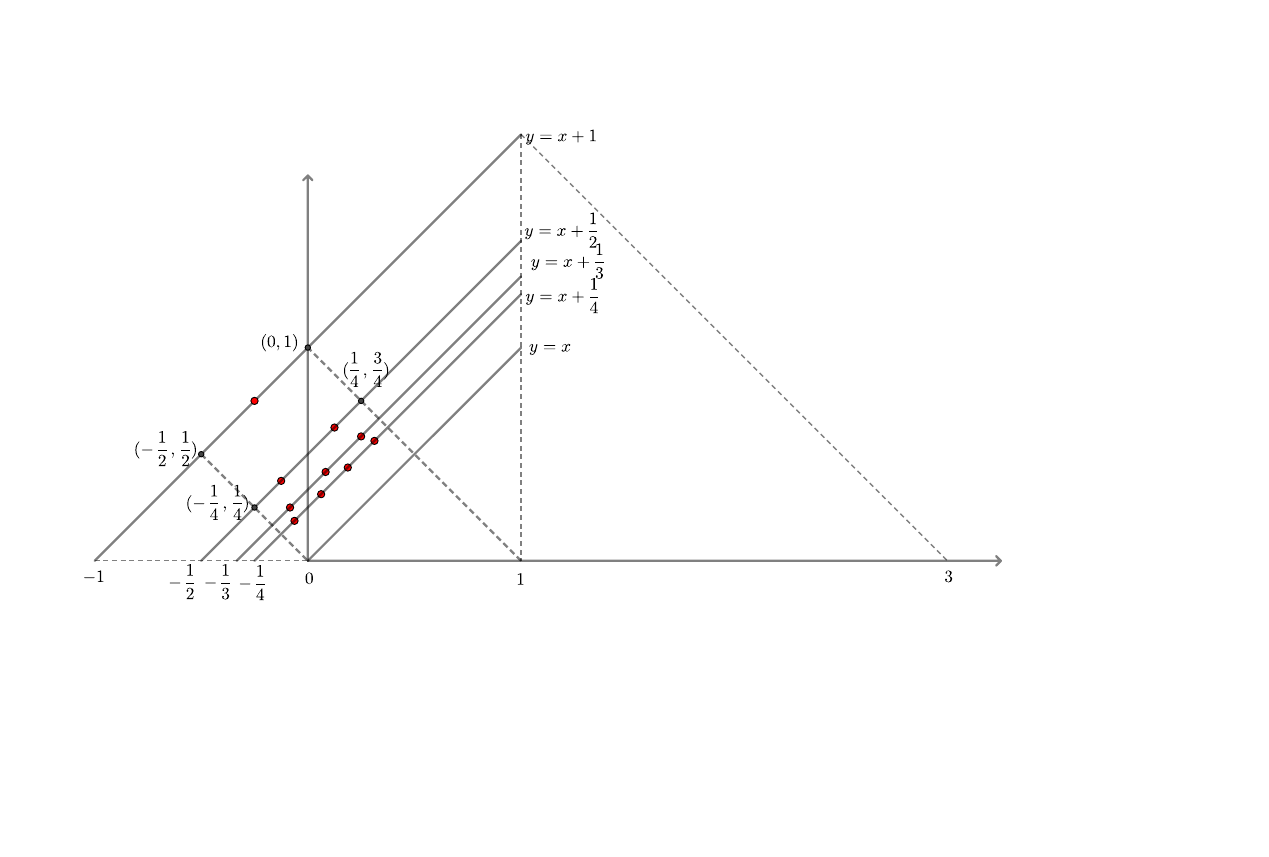}}
		\vspace{-1.2 in}
		\caption{Points in the optimal sets of $n$-points for $1\leq n\leq 4$.} \label{Fig}
	\end{figure}

	The following theorem gives the constrained optimal set of $n$-points for $\nu$. 
	\begin{theorem} \label{theo1} 
		A constrained optimal set of $n$-points for the probability distribution $\nu$ is given by 
		\[\left\{\Big(\frac{1}{2} \left(\frac{2 j-1}{2 n}-\frac{1}{n}\right),\frac{1}{2} \left(\frac{2 j-1}{2 n}-\frac{1}{n}\right)+\frac{1}{n}\Big) : 1\leq j\leq n\right\},\]
		with $n$th constrained quantization error 
		\[R_n=\frac{4 n^2+12 n+13}{24 n^2}.\]
	\end{theorem} 
	
	\begin{proof} Let $\ga_n:=\set{(a_j, b_j) : 1\leq j\leq n}$ be a constrained optimal set of $n$-points for $\nu$ such that $a_1<a_2<\cdots<a_n$.  By Lemma~\ref{lemmaM1}, we know that $U_n(\ga_n)$ is an optimal set of $n$-means for $\nu$, i.e., 
		\[U_n(\ga_n)=\left\{\frac {2j-1}{2n} : 1\leq j\leq n\right\}.\]
		Since $U_n$ is an injective function, we have 
		\[\ga_n=U_n^{-1}\Big\{\frac {2j-1}{2n} : 1\leq j\leq n\Big\}=\Big\{U_n^{-1}\Big(\frac {2j-1}{2n}\Big) : 1\leq j\leq n\Big\}\]
		i.e., 
		\[\ga_n=\left\{\Big(\frac{1}{2} \left(\frac{2 j-1}{2 n}-\frac{1}{n}\right),\frac{1}{2} \left(\frac{2 j-1}{2 n}-\frac{1}{n}\right)+\frac{1}{n}\Big) : 1\leq j\leq n\right\}.\]
		Writing 
		\[(a_j, b_j)=\Big(\frac{1}{2} \left(\frac{2 j-1}{2 n}-\frac{1}{n}\right),\frac{1}{2} \left(\frac{2 j-1}{2 n}-\frac{1}{n}\right)+\frac{1}{n}\Big),\]
		for $1\leq j\leq n$, we have the $n$th constrained quantization for $n$-points as 
		\begin{align*}
			&R_n=\int_{\D R} \min_{a\in \ga_n}\rho(x, a)d\nu(x)\\
			&=\int_0^{a_1+a_2+\frac 1 n}   \rho(x, (a_1, a_1+\frac 1n)) \, dx+\sum _{i=2}^{n-1} \int_{a_{i-1}+a_i+\frac 1 n}^{a_{i}+a_{i+1}+\frac 1 n}\rho(x, (a_i, a_i+\frac 1 n))\, dx\\
			&\qquad +\int_{a_{n-1}+a_n+\frac 1 n}^1\rho(x, (a_n, a_n+\frac 1 n)) \, dx,
		\end{align*}
		which upon simplification yields
		\[R_n=\frac{4 n^2+12 n+13}{24 n^2}.\]
		Thus, the proof of the theorem is complete (see Figure~\ref{Fig}). 
	\end{proof}

	\subsection{Constrained quantization dimension and constrained quantization coefficient} \label{sec4}
	In this subsection, we show that the constrained quantization dimension $D(\nu)$ exists and equals two. We further show that the $D(\nu)$-dimensional constrained quantization coefficient for $\nu$ exists as a finite positive number.

	\begin{theorem}\label{theo2} 
		The constrained quantization dimension $D(\nu)$ of the probability measure $\nu$ exists, and $D(\nu)=2$. 
	\end{theorem}
	
	\begin{proof} By Theorem~\ref{theo1}, the $n$th constrained quantization error is given by 
		\[R_n=\frac{4 n^2+12 n+13}{24 n^2}.\]
		Notice that $R_\infty=\mathop{\lim}\limits_{n\to \infty} R_n=\frac 1 6$.
		Hence, the constrained quantization dimension is given by 
		\[D(\nu)=\lim_{n\to \infty} \frac{2\log n}{-\log(R_n-R_\infty)}=\lim_{n\to \infty}\frac{2 \log n}{-\log \left(\frac{4 n^2+12 n+13}{24 n^2}-\frac{1}{6}\right)}=2,\]
		which is the theorem. 
	\end{proof}

	\begin{theorem} \label{theo3} 
		The $D(\nu)$-dimensional constrained quantization coefficient for $\nu$ exists, and equals $\frac 12$. 
	\end{theorem}
	\begin{proof}
		We have 
		\[R_n=\frac{4 n^2+12 n+13}{24 n^2} \te{ and } R_\infty=\mathop{\lim}\limits_{n\to \infty} R_n=\frac 1 6,\]
		and hence, using \eqref{eq00100}, we have the $D(\nu)$-dimensional constrained quantization coefficient as 
		\[\lim_{n\to \infty} n(R_n-R_\infty)=\lim_{n\to \infty} n \left(\frac{4 n^2+12 n+13}{24 n^2}-\frac{1}{6}\right)=\frac 12.\]
		Thus, the proof of the theorem is complete. 
	\end{proof} 
	
	\begin{remark} \label{rem1} 
		It is interesting to observe that although the constrained quantization dimension $D(\nu)$ of the uniform distribution with respect to the family of constraints $\set{C_j : j\in \D N}$ exists, it does not coincide with the Euclidean dimension of the space in which the support of the probability measure is defined. On the other hand, it is well-known that for an absolutely continuous probability measure, the unconstrained quantization dimension always exists and equals the Euclidean dimensions of the underlying space in which the support of the measure lies (see \cite{BW}).
	\end{remark}

	\section{Constrained quantization with respect to the family of constraints $C_j:= \set{(x,y) : x^2+y^2= \frac{1}{j^2}}~ \text{\normalfont for } j\in \D N.$} \label{section4}
	
	In this section, we determine the constrained optimal sets of $n$-points and the corresponding $n$th constrained quantization errors for the uniform distribution $\nu$ with support the closed interval $[0, 1]$ for all $n \in \mathbb{N}$ with respect to the family of constraints 
	\[
	C_j:= \bigl\{ (x,y) : x^2+y^2= \tfrac{1}{j^2} \bigr\}, \quad j \in \mathbb{N}.
	\]
	Notice that for any $j\in \D N$, the circles $C_j$ intersect the closed interval $[0, 1]$ at the points $(\frac 1 j, 0)$. 
	\begin{prop} \label{prop1}
		The constrained optimal set of one-point is given by $\ga_1=\set{(\frac 12, 0)}$ with constrained quantization error $R_1=\frac 1{12}$. 
	\end{prop}
	\begin{proof}
		Let $\ga_1$ be an optimal set of one-point, and it contains the element from the circle $C_t$ for some $t\in \D N$. Any element on the circle $C_t$ can be written as $(\frac 1 t\cos \gq, \frac 1 t\sin \gq)$, where $0\leq \gq< 2\pi$. The distortion error due to the set $\set{(\frac 1 t\cos \gq, \frac 1 t\sin \gq)}$ is given by 
		\[R(\nu;\set{(\frac 1 t\cos \gq, \frac 1 t\sin \gq)})=\int\rho(x, (\frac 1 t\cos \gq, \frac 1 t\sin \gq))\,dx=\frac{1}{t^2}-\frac{\cos \theta }{t}+\frac{1}{3}, \]
		which is minimum if $t=2$ and $\gq=0$, and the minimum value is $\frac 1{12}$. Hence, an optimal set of one-point is given by $\ga_1=\set{(\frac 12, 0)}$ with constrained quantization error $R_1=\frac 1{12}$. Thus, the proof of the proposition is complete. 
	\end{proof} 
	
	\begin{lemma} \label{lemma1}
		Let $\ga_n$ be a constrained optimal set of $n$-points for $\nu$ for any $n \in \D N$. Then, $\ga_n \ci \te{supp}(\nu)$, i.e., for any $a\in \ga_n$ we have $a\in [0, 1]$.  
	\end{lemma} 
	\begin{proof}
		Let $\ga_n$ be a constrained optimal set of $n$-points for $\nu$ for any $n \in \D N$. If $n=1$, then by Proposition~\ref{prop1}, we have $\ga_n=\set{(\frac 12, 0)}\ci \te{supp}(\nu)$, i.e., the lemma is true for $n=1$. Let us now prove the lemma for any $n\geq 2$. Let $a\in \ga_n$. We need to show that $a\in [0, 1]$. Since $a\in \ga_n$, we have $a\in C_t$ for some $t\in \D N$, and hence, we can write 
		$a:=(\frac 1t \cos \gq, \frac 1t \sin \gq)$, where $0\leq \gq<2\pi$. Let the Voronoi region of $a$ cuts the support of $\nu$ at the elements $(c, 0)$ and $(d, 0)$, where $0\leq c<d\leq 1$. Now, the distortion error contributed by the element $(\frac 1t \cos \gq, \frac 1t \sin \gq)$ on its own Voronoi region is given by 
		\[\int_{c}^d \rho(x, (\frac 1t \cos \gq, \frac 1t \sin \gq)) \, dx=\frac{-c^3 t^2+3 t \left(c^2-d^2\right) \cos \theta -3 c+d^3 t^2+3 d}{3 t^2},\]
		which is minimum if $\cos\gq=1$, i.e., if $\gq=0$, which yields the fact that 
		$a=(\frac 1t, 0)$, and hence $a\in[0, 1]$ for $n\geq 2$. Thus, with the help of Proposition~\ref{prop1}, the proof of the lemma is complete. 
	\end{proof}

	\begin{prop} \label{prop2}
		The constrained optimal set of two-points is given by $\ga_2=\set{(\frac 14, 0), (1, 0)}$ with constrained quantization error $R_2=\frac{31}{768}$. 
	\end{prop}
	\begin{proof}
		Let $\ga_2$ be a constrained optimal set of two-points with respect to the family of constraints $C_j$ for $j\in \D N$. Then, there exists $a_2, a_1\in \D N$ with $a_2>a_1$ such that 
		$\ga_2=\set{\frac 1 {a_2}, \frac 1 {a_1}}$. The distortion error is given by 
		\begin{align*} 
			R(\nu;\ga_2)&=\int\min_{a \in \ga_2}\rho(x, a) \, d\nu(x)=\int_0^{\frac 12(\frac 1 {a_1}+\frac 1  {a_2})}\rho(x, (\frac 1 {a_2}, 0))\ d\nu +\int_{\frac 12(\frac 1 {a_1}+\frac 1 {a_2})}^1\rho(x, (\frac 1 {a_1}, 0))\ d\nu \\
			&=\frac{1}{a_1^2}-\frac{1}{4 a_1^3}-\frac{1}{4 a_1^2 a_2}+\frac{1}{4 a_1 a_2^2}+\frac{1}{4 a_2^3}-\frac{1}{a_1}+\frac{1}{3},
		\end{align*}
		which is minimum if ${a_2}=4$ and $a_1=1$, and the minimum values is $\frac{31}{768}$. Hence, the constrained optimal set of two-points is given by $\ga_2=\set{(\frac 14, 0), (1, 0)}$ with constrained quantization error $R_2=\frac{31}{768}\approx 0.0403646$. 
	\end{proof} 
	\begin{prop} \label{prop3}
		The constrained optimal set of three-points is given by $\ga_3=\set{(\frac 16, 0), (\frac 12, 0), (1, 0)}$ with constrained quantization error $R_3=\frac{13}{864}$. 
	\end{prop}
	\begin{proof}
		Let $\ga_3$ be a constrained optimal set of three-points with respect to the family of constraints $C_j$ for $j\in \D N$. Then, there exists $a_3, a_2, a_1\in \D N$ with $a_3>a_2>a_1$ such that 
		$\ga_3=\set{\frac 1 {a_3}, \frac 1 {a_2}, \frac  1{a_1}}$. The distortion error is given by 
		\begin{align*} 
			R(\nu;\ga_3)&=\int\min_{a \in \ga_3}\rho(x, a) \, d\nu(x)\\
			&=\int_0^{\frac 12(\frac 1 {a_3}+\frac 1 {a_2})}\rho(x, (\frac 1 {a_3}, 0))\ d\nu +\int_{\frac 12(\frac 1 {a_3}+\frac 1 {a_2})}^{\frac 12(\frac 1 {a_2}+\frac 1 {a_1})}\rho(x, (\frac 1 {a_2}, 0))\ d\nu  +\int_{\frac 12(\frac 1 {a_2}+\frac 1 {a_1})}^1\rho(x, (\frac 1 {a_1}, 0))\ d\nu \\
			&=\frac{1}{a_1^2}-\frac{1}{4 a_1^3}-\frac{1}{4 a_1^2 a_2}+\frac{1}{4 a_1 a_2^2}-\frac{1}{4 a_2^2 a_3}+\frac{1}{4 a_2 a_3^2}-\frac{1}{4 a_3^2 a_4}+\frac{1}{4 a_3 a_4^2}+\frac{1}{4 a_4^3}-\frac{1}{a_1}+\frac{1}{3},
		\end{align*}
		which is minimum if $a_3=6$, $a_2=2$ and $a_1=1$, and the minimum values is $\frac{13}{864}$. Hence, the constrained optimal set of three-points is given by $\ga_3=\set{(\frac 16, 0), (\frac 12, 0), (1, 0)}$ with constrained quantization error $R_2=\frac{13}{864}\approx 0.0150463$. 
	\end{proof} 
	
	Proceeding in the similar way as Proposition~\ref{prop3}, the following proposition can be proved. 
	\begin{prop} \label{prop4}
		If $\ga_4$ is the constrained optimal set of four-points and $R_4$ is the corresponding constrained quantization error, then $\ga_4=\set{(\frac 1{10}, 0), (\frac 13, 0), (\frac 12, 0), (1, 0)}$ and $R_4=\frac{439}{36000}$. 
	\end{prop}
	
	The following theorem is a special case of Theorem~2.1.1 that appears in \cite{RR}. 
	\begin{theorem} (see \cite[Theorem~2.1.1]{RR}) \label{theorem0}   Let $Q$ be a uniform distribution on the closed interval $[0, 1]$. Then, the unconstrained optimal set $n$-means is given by $ \gb_n:=\set{\frac {2i-1}{2n}: 1\leq i\leq n}$, and the corresponding $n$th unconstrained quantization error is
		$R_n(Q)=\frac{(a-b)^2}{12 n^2}.$
	\end{theorem}

	\begin{remark} \label{remMe} 
		Let $k\in \D N$. Then, a set $A_k$ is a constrained set containing $k$ elements, by that it is meant that 
		\[A_k \ci \UU_{j\in \D N} C_j  \te{ with } \te{card}(A_k)=k.\]
		Let $\gb_n$ be an unconstrained optimal set of $n$-means, where $n\in\D N$. A constrained set $A_k$ corresponds to $\gb_n$, by that it is meant that $A_k$ is the collection of all elements in $\uu_{j\in \D N} C_j$, which are nearest to the elements in the set $\gb_n$. Let $\ga_k$ denote a constrained optimal set of $k$-points for $k\in \D N$. By Propositions~\ref{prop1} to \ref{prop4}, we see that $\ga_1$ corresponds to $\gb_1$, $\ga_2$ corresponds to $\gb_2$, $\ga_3$ corresponds to $\gb_3$, and  $\ga_4$ corresponds to $\gb_4$. Notice that for a given $k\in \D N$, there can be multiple $A_k$ corresponding to $\gb_n$ for different values of $n\in \D N$. For example, the unconstrained optimal sets $\gb_7$ and $\gb_8$ of seven- and eight-means are given by 
		\[\gb_7=\set{\frac{1}{14},\frac{3}{14},\frac{5}{14},\frac{1}{2},\frac{9}{14},\frac{11}{14},\frac{13}{14}} \te{ and } \gb_8=\set{\frac{1}{16},\frac{3}{16},\frac{5}{16},\frac{7}{16},\frac{9}{16},\frac{11}{16},\frac{13}{16},\frac{15}{16}}. \] 
		The corresponding constrained sets, respectively, are obtained as 
		\[\set{1,\frac{1}{2},\frac{1}{3},\frac{1}{5},\frac{1}{14}} \te{ and } \set {1,\frac{1}{2},\frac{1}{3},\frac{1}{5},\frac{1}{16}},\]
		both of which contain the same number of elements from the family of constraints 
		with distortion errors $\frac{9301}{823200}$ and $\frac{13883}{1228800}$, respectively. Thus, to calculate a constrained optimal set $\ga_k$ of $k$-points, we will first calculate all constrained sets $A_k$, then the constrained set $A_k$ for which the distortion error is smallest will be the constrained optimal set $\ga_k$ of $k$-points. 
		Below we give a methodology how to calculate the constrained optimal sets $\ga_k$ of $k$-points and the corresponding quantization errors for any positive integer $k\in \D N$.  \qed
	\end{remark} 
	
	\begin{method} \label{method1}
		The computational technique given here is developed based on the knowledge obtained from Remark~\ref{remMe}. 
		Let $\gb_n$ represent an unconstrained optimal set of $n$-means, i.e., by Theorem~\ref{theorem0}, we have 
		\[\gb_n=
		\left\{ \,\frac{2j - 1}{2n} \;\middle|\; j = 1,2,\cdots,n \,\right\}.
		\]
		Suppose we need to calculate a constrained optimal set $\ga_m$ of $m$-points. Let $\C P(\D N):=\set{A : A \ci \D N}$ be the power set of $\D N$. Let $G(m): \D N \to \C P(\D N)$ be a set-valued function which gives the collections of all positive integers $n$, such that a constrained set corresponding to $\gb_n$ contains exactly $m$ elements.  Then, 
		\begin{equation} \label{Pav1}
			\begin{aligned}
				G(m) :=
				\begin{cases}
					\left\{\, k \in \mathbb{Z} : 
					\left\lceil \tfrac{m^{2}}{4} \right\rceil \leq k \leq 
					\left\lfloor \tfrac{(m+1)^{2}}{4} \right\rfloor \,\right\}, & \text{if $m$ is even}, \\[1em]
					\left\{\, k \in \mathbb{Z} : 
					\left\lceil \tfrac{m^{2}}{4} \right\rceil \leq k \leq 
					\left\lfloor \tfrac{(m+1)^{2}}{4} - 1 \right\rfloor \,\right\}, & \text{if $m$ is odd}.
				\end{cases}
			\end{aligned}
		\end{equation} 
		Let $\te{card}(G(m))=\ell$, where $\ell\in \D N$. Then, we can write 
		\[G(m)=\set{n_j : 1\leq j\leq \ell}.\]
		Let $A_{m_j}$ be a constrained set containing $m$ elements corresponding to the unconstrained optimal set $\gb_{n_j}$ of $n_j$-means for $1\leq j\leq \ell$. Let $R(\nu;A_{m_j})$ represent the distortion errors for the probability distribution $\nu$ with respect to the sets $A_{m_j}$ for $1\leq j\leq \ell$. 
		Let \[R_m(\nu)=\min\set{R(\nu;A_{m_j}) : 1\leq j\leq \ell}.\]
		Each of the sets $A_{m_j}$ for which $R_m(\nu)=R(\nu;A_{m_j})$, gives a constrained optimal set $\ga_m$ of $m$-points with constrained quantization error $R_m(\nu)$. 
		Notice that 
		\begin{equation} \label{Pav2}
			A_{m_j}=\left\{ \,\frac{1}{j} : \; 
			j \in \bigcup_{j \in\gb_{n_j}} \Bigl\{ \operatorname{Round}\!\left(\tfrac{1}{j}\right) \Bigr\} \,\right\},
		\end{equation} 
		where the function $\te{Round}(x)$ rounds the number $x$ to the nearest integer. 
		As $A_{m_j}$ contains $m$ elements, we can write
		\[A_{m_j}:=\set{\frac 1{a^{(j)}_m}, \cdots, \frac 1{a^{(j)}_3}, \frac 1{ a^{(j)}_2}, \frac 1{a^{(j)}_1}}\]
		where $\frac 1{a^{(j)}_m}< \cdots<\frac 1{a^{(j)}_3}<\frac 1{ a^{(j)}_2}<\frac 1{a^{(j)}_1}$ and $1\leq j\leq \ell$, 
		and then we can write
		
		\begin{equation}  \label{Pav3}
			\begin{aligned}
				R(\nu;A_{m_j})&=\int_{\frac 12(\frac 1{a^{(j)}_1}+\frac 1{a^{(j)}_2 })}^{1} \rho(x, (\frac 1{a^{(j)}_1}, 0))^2\, d\nu +
				\sum_{k=2}^{m-1} \int_{\frac 12(\frac 1{a^{(j)}_{k+1} }+\frac 1{a^{(j)}_k})}^{\frac 12(\frac 1{a^{(j)}_k}+\frac 1{a^{(j)}_{k-1}})} \rho(x, (\frac 1{a^{(j)}_k}, 0))^2\, d\nu\\
				&\qquad \qquad  +\int_0^{\frac 12(\frac 1{a^{(j)}_m}+\frac 1{a^{(j)}_{m-1} })}  \rho(x, (\frac 1{a^{(j)}_m}, 0))^2\, d\nu.
			\end{aligned}
		\end{equation} \qed
	\end{method} 
	The above methodology is illustrated by the following examples. 
	\begin{example} \label{exam1}
		To obtain a constrained optimal set of nine-points, we first calculate $G(m)$, where $m=9$. By \eqref{Pav1}, we have 
		\[G(m)=\{21, 22, 23, 24\}.\]
		Since $\te{card}(G(m))=4$, we have 
		\[n_1=21, \, n_2=22, \,n_3=23, \, n_4=24.\]
		Using \eqref{Pav2} and \eqref{Pav3}, we have 
		\begin{align*}
			A_{m_1}&= \left\{\frac{1}{42},\frac{1}{14},\frac{1}{8},\frac{1}{6},\frac{1}{5},\frac{1}{4},\frac{1}{3},\frac{1}{2},1\right\} \te{ with } R(\nu;A_{m_1})=\frac{322921}{29635200}\approx 0.0108965,\\
			A_{m_2}&= \left\{\frac{1}{44},\frac{1}{15},\frac{1}{9},\frac{1}{6},\frac{1}{5},\frac{1}{4},\frac{1}{3},\frac{1}{2},1\right\} \te{ with } R(\nu;A_{m_2})=\frac{7518677}{689990400}\approx 0.0108968,\\
			A_{m_3}&= \left\{\frac{1}{46},\frac{1}{15},\frac{1}{9},\frac{1}{7},\frac{1}{5},\frac{1}{4},\frac{1}{3},\frac{1}{2},1\right\} \te{ with } R(\nu;A_{m_3})=\frac{315753139}{28974493800}\approx 0.0108976,\\
			A_{m_4}&= \left\{\frac{1}{48},\frac{1}{16},\frac{1}{10},\frac{1}{7},\frac{1}{5},\frac{1}{4},\frac{1}{3},\frac{1}{2},1\right\} \te{ with } R(\nu;A_{m_4})=\frac{5904907}{541900800}\approx 0.0108967.
		\end{align*} 
		Thus, we see that $R_m(\nu)=\frac{322921}{29635200}$. Hence, the constrained optimal set of nine-points is given by  $\ga_9=\left\{\frac{1}{42},\frac{1}{14},\frac{1}{8},\frac{1}{6},\frac{1}{5},\frac{1}{4},\frac{1}{3},\frac{1}{2},1\right\}$  with constrained quantization error $R_9(\nu)=\frac{322921}{29635200}$. 
	\end{example}
	\begin{example} \label{exam2}
		To obtain a constrained optimal set of twenty-points, we first calculate $G(m)$, where $m=20$. By \eqref{Pav1}, we have 
		\[G(20)=\{100, 101, 102, 103, 104, 105, 106, 107, 108, 109, 110\}.\]
		Since $\te{card}(G(m))=11$, we have 
		\[n_1=100, \, n_2=101, \,n_3=102, \, \cdots, \, n_{11}=110.\]
		Using \eqref{Pav2} and \eqref{Pav3}, we calculate all $A_{m_j}$ and $R(\nu;A_{m_j})$.  
		Thus, we see that 
		\[R_m(\nu)=\min\set{R(\nu;A_{m_j}) : 1\leq j\leq 11}=\frac{10205531205421}{939137753917440}=R(\nu;A_{m_6}).\] Hence, the constrained optimal set of twenty-points is given by  
		\[\ga_{20}=\left\{\frac{1}{210},\frac{1}{70},\frac{1}{42},\frac{1}{30},\frac{1}{23},\frac{1}{19},\frac{1}{16},\frac{1}{14},\frac{1}{12},\frac{1}{11},\frac{1}{10},\frac{1}{9},\frac{1}{8},\frac{1}{7},\frac{1}{6},\frac{1}{5},\frac{1}{4},\frac{1}{3},\frac{1}{2},1\right\}\] 
		with constrained quantization error $R_{20}(\nu)=\frac{10205531205421}{939137753917440}$. 
	\end{example}
	\begin{remark}
		Proceeding in a similar way, one can calculate the constrained optimal sets of $n$-points for all $n\geq 3$. 
	\end{remark} 
	
	\section{Conclusion and Future Work}\label{conclusionfuturework}
	In this paper, we investigated constrained quantization for a uniform probability distribution with respect to two distinct families of constraints: a family of straight lines and a family of concentric circles. For the family of straight lines, we explicitly determined the constrained optimal sets of $n$-points and the corresponding $n$th constrained quantization errors for all $n \in \mathbb{N}$. Furthermore, we established the existence of the constrained quantization dimension and showed that it equals two, together with the determination of the associated constrained quantization coefficient. For the family of concentric circles, we developed a systematic methodology to compute constrained optimal sets of $n$-points and the corresponding constrained quantization errors. Several explicit examples were provided to demonstrate the application of this methodology.  
	
	The results highlight the fact that constrained quantization can exhibit significant structural differences compared to the unconstrained case. In particular, constrained optimal sets may have fewer elements than the prescribed number $n$, and the constrained quantization dimension need not coincide with the Euclidean dimension of the underlying space. These findings emphasize the intricate nature of constrained quantization and its potential for modeling problems where symmetry, geometry, or external restrictions must be taken into account.  
	
	Several avenues for future research naturally emerge from this study. One direction is to extend the analysis to other families of constraints, such as polygons, ellipses, or more general algebraic curves. Another promising line of inquiry is to investigate constrained quantization for higher-dimensional probability distributions, where geometric restrictions often arise in applications. Additionally, the interplay between constrained quantization and dynamical systems, particularly in the context of invariant measures, could provide a deeper theoretical understanding. From a computational perspective, further refinement of algorithms for determining constrained optimal sets in complex settings would also be of considerable interest.  
	
	Overall, the framework developed here not only advances the theory of constrained quantization for uniform distributions but also lays the groundwork for broader exploration in probability, geometry, and applied fields such as information theory, signal processing, and mathematical modeling of physical systems.  
	\qed 
	
	\section*{Declaration}
	
	\noindent
	\textbf{Conflicts of interest.} We do not have any conflict of interest.\\
	\\
	\noindent
	\textbf{Data availability:} No data were used to support this study.\\
	\\
	\noindent
	\textbf{Code availability:} Not applicable\\
	\\
	\noindent
	\textbf{Authors' contributions:} Each author contributed equally to this manuscript.

	\bibliographystyle{amsplain}
	\bibliography{References}

	%\end{document}
\end{document}